\documentclass[11pt,a4paper]{article}

\usepackage{amsmath,amsfonts,amssymb,amsthm,amsrefs,cite}
\usepackage{mathrsfs}

\evensidemargin=\oddsidemargin
\advance\topmargin by -\headheight
\advance\topmargin by -\headsep

\newtheorem{thm}{Theorem}[section]
\newtheorem{lem}[thm]{Lemma}
\newtheorem{cly}[thm]{Corollary}
\newtheorem{prop}[thm]{Proposition}

\theoremstyle{definition}
\newtheorem{defn}[thm]{Definition}

\theoremstyle{remark}
\newtheorem{rem}[thm]{Remark}

\numberwithin{equation}{section}        %
\newcommand{\N}{\mathbb{N}}             %

\newcommand{\chrom}[1]{\chi_{#1}}

\makeatletter
\renewcommand{\sideset}[3]{%
  \@mathmeasure\z@\displaystyle{#3}%
  \global\setbox\@ne\vbox to\ht\z@{}\dp\@ne\dp\z@
  \setbox\tw@\box\@ne
  \@mathmeasure4\displaystyle{\copy\tw@#1}%
  \@mathmeasure6\displaystyle{#3\nolimits#2}%
  \dimen@-\wd6 \advance\dimen@\wd4 \advance\dimen@\wd\z@
  \mathop{\hbox to\dimen@{}}\!%
  \mathop{\kern-\dimen@\box4\box6}%
}
\makeatother

\newcommand{\sidesum}[1]{\sideset{}{^{(#1)}}\sum}

\usepackage{authblk}

\title{Recursion relations for chromatic coefficients\\ for graphs and hypergraphs}
\author[1]{Bergfinnur Durhuus\thanks{durhuus@math.ku.dk}}
\author[1,2]{Angelo Lucia\thanks{alucia@caltech.edu}}
\affil[1]{Department of Mathematical Sciences, Copenhagen University \protect\\
Universitetsparken 5, DK-2100 Copenhagen {\O}, Denmark}
\affil[2]{Walter Burke Institute for Theoretical Physics and Institute for Quantum Information \& Matter,
California Institute of Technology, Pasadena, CA 91125, USA}

\begin{document}
\maketitle

\begin{abstract}
    We establish a set of recursion relations for the coefficients in the chromatic polynomial of a graph or a hypergraph.
    As an application we provide a generalization of Whitney's broken cycle theorem for hypergraphs,
    as well as deriving an explicit formula for the linear coefficient of the chromatic polynomial
    of the $r$-complete hypergraph in terms of roots of the Taylor polynomials for the exponential function.
\end{abstract}

\section{Introduction}
The chromatic polynomial $\chi_G$ associated to a graph $G$, introduced by Birkhoff \cite{birkhoff},
is determined by defining $\chi_G(\lambda)$, for $\lambda\in \mathbb N$, to be the number of
colourings of the vertices of $G$ with at most $\lambda$ colours, such that no adjacent vertices are
attributed the same colour\cite{read,Dong2005}. The definition extends to hypergraphs\cite{bujitas2015}, by considering colourings such that each hyperedge contains at least two vertices with different colours.
In the case of graphs, Whitney's broken cycle theorem \cite{whitney,Dohmen1999,blass-sagan,dohmen2} provides a combinatorial interpretation to the coefficients of the chromatic polynomial $\chi_G(\lambda)$ : if a graph $G$ has $n$ vertices, then the coefficient of $\lambda^i$ is given, up to the sign $(-1)^{n-i}$, by the number of spanning subgraphs of $G$ with $n-i$ edges with the property of not containing as a subset any of a particular list of special subgraphs of $G$, known as \emph{broken cycles}\footnote{Whitney's original theorem mentions \emph{broken circuits} instead, but the distinction between circuits and cycles is not relevant in this context.}.

In the present article, we establish a set of recursion relations for the coefficients of the chromatic polynomial of a graph or hypergraph, which allow us to express the $i$-th order coefficient in terms of products of linear coefficients of certain subgraphs. We similarly show that the combinatorial quantities appearing in Whitney's theorem (as well as a natural generalization of them which covers the case of hypergraphs) also satisfy the same recursion relations (up to a sign factor).
Since the two sequences are recursively defined by the same relations and it can be easily verified
that they coincide on empty graphs, we obtain as a consequence a generalization of the broken cycle
theorem for hypergraphs. There are a number of different extensions of Whitney's theorem to
hypergraphs already present in the literature \cite{dohmen1, dohmen2, trinks, dohmen-trinks}. The one we present here encompasses those known to us.

As a second application of the recursion relations, we derive an explicit formula for the linear chromatic coefficient of the $r$-complete hypergraphs in terms of the roots of the $(r-1)$'th Taylor polynomial of the exponential function (where the $r$-complete hypergraph is the hypergraph containing all possible hyperedges of cardinality $r$).

Whitney's theorem implies that the coefficients of the chromatic polynomial of a graph are always integers with alternating signs. Moreover, applying the deletion-contraction principle for the chromatic polynomial \cite{birkhoff,whitney,Dong2005}, one can also show that they are numerically upper bounded by the corresponding coefficient for the complete graph of the same order. We show that both these facts can be obtained in a simple way as a consequence of the recursion relations we present, without using neither Whitney's theorem nor the deletion-contraction principle.

The paper is organized as follows. In Section~\ref{sec:graphs}, we start by presenting the simpler case of the recursion relations for graphs, together with a new proof of Whitney's theorem in its original form. The section will follow the same approach we will use for the general case, but since it is arguably easier we present it here for illustration of the method, but it can safely be skipped.
In Section~\ref{sec:hypergraphs} we present the general case of hypergraphs, and the generalization of Whitney's theorem.
Finally in Section~\ref{sec:complete-hypergraphs} we apply the recursion relations to obtain the formula for the linear coefficient of the $r$-complete hypergraph.

\section{The recursion relations for graphs}
\label{sec:graphs}
In this section $G=(V,E)$ denotes a simple graph, where $V$ is a non-empty finite set and $E$ is a
set of unordered pairs of elements in $V$. The members in $V$ and $E$ are called the \emph{vertices}
and \emph{edges} in $G$, respectively. The order of $G$, i.e.\ the number of vertices $|V|$, will be denoted by $n$. By $k(G)$ we shall denote the number of connected components of $G$. If $F\subseteq E$, the graph $\bar G\langle F\rangle \equiv (V,F)$ is called  the \emph{spanning subgraph} of $G$ induced by $F$, and we shall write $k(F)$ for $k(\bar G\langle F\rangle)$. If $V'\subset V$, the graph $(V',E')$ where $E' = \{\{x,y\}\in E \mid x,y\in V'\}$ is called the \emph{subgraph of $G$ induced by $V'$}. It will be denoted by $G[V']$.

\begin{defn}
Let $\lambda\in\mathbb N$. A $\lambda$-colouring of a graph $G=(V,E)$ is a map $\pi: V\to \{1,2,\dots,\lambda\}$. A $\lambda$-colouring is called \emph{proper} if for each edge $e=\{x,y\}\in E$ it holds that $\pi(x)\neq \pi(y)$. We define
$\chrom{G}(\lambda)$ to be the number of proper $\lambda$-colourings of $G$.
\end{defn}

It is well kown that the $\chi_G(\lambda)$ is a polynomial in $\lambda$.
\begin{thm}\label{chrom1}
The function $\chrom{G}$ is a polynomial, called the \emph{chromatic polynomial} of $G$, given by
$$
\chrom{G}(\lambda) = \sum_{i=1}^n a_i(G) \lambda^{i}\,,
$$
where
\begin{equation}\label{coeff}
  a_i(G) = \sum_{\substack{F\subseteq E \\ k(F) = i}} (-1)^{|F|}\,.
\end{equation}
\end{thm}
\begin{proof}  Define for any edge $e\in E$ the function $f_e$ on the set of colourings of $G$ by
$$
f_e(\pi) = \begin{cases} 0\;\mbox{if $\pi$ is constant on $e$}\\ 1\; \mbox{otherwise.}\end{cases}\,.
$$
Then
\begin{align*}
\chrom{G}(\lambda) = \sum_\pi \prod_{e\in E} f_e(\pi)
= \sum_\pi \prod_{e\in E} (1-(1-f_e(\pi))
&= \sum_\pi \sum_{F\subseteq E} (-1)^{|F|}\prod_{e\in F} (1-f_e(\pi))\\
&= \sum_{F\subseteq E} (-1)^{|F|}\lambda^{k(F)}\,.
\end{align*}
\end{proof}

Whitney refined this result in what is known as his \emph{broken-cycle theorem} \cite{whitney}. Let $\leq$ be an arbitrary linear ordering of the edge set $E$. A \emph{broken cycle} of $G$ is then a set of edges $F\subseteq E$ obtained by removing the maximal edge from a cycle of $G$.
\begin{thm}[Whitney 1932]\label{thm:whitney} For $i=1,\dots, n$ we have that
\begin{equation}
  a_i(G) = (-1)^{n-i} h_i(G)\,,
\end{equation}
where $h_i(G)$ is the number of spanning subgraphs of $G$ with $n-i$ edges and containing no broken cycle.
\end{thm}

We will establish, in the next three lemmas, a set of recursion relations for coefficients $a_i$ and
for coefficients $h_i$, respectively. Up to a sign factor, both sets of coefficients will be shown
to satisfy the same recursion relations, and by observing that they coincide on the empty graph we will obtain as a consequence an inductive proof of Theorem~\ref{thm:whitney}.

Recall, that an edge $e\in E$ is called a \emph{bridge} in $G=(V,E)$ if $k(E) < k(E\setminus{e})$
(i.e.\ removing $e$ increases the number of connected components of the graph), in which case we
must have $k(E\setminus{e}) = k(E)+1$. If $F\subseteq E$ we say that $e\in F$ is a bridge in $F$ if
it is a bridge in ${\bar G}\langle F\rangle$. We denote by ${\cal B}_e^i$ the collection of $F\subseteq E$ such that $e$ is a
bridge in $F$ and $k(F)=i$.

\begin{lem}\label{lemma:subgraph}
 Let $G=(V,E)$ be a graph with $E\neq \emptyset$ and fix $e\in E$. We have that
  \begin{equation}\label{eq:decomposition-a}
    a_i(G) = b_e^{i}(G) - b_e^{i-1}(G)\,,
  \end{equation}
  where $b_e^0(G) = 0$ and
	\[ b^i_e(G) = \sum_{F\in {\cal B}_e^i} (-1)^{|F|}\,, \quad \forall i\ge 1.\]
      \end{lem}
\begin{proof}
For each subset $F$ of $E$ exactly one of the following holds:
\[
   1)\; e\notin F,\qquad 2)\; e\; \mbox{is a bridge in $F$},\qquad 3)\; e\in F,\, \mbox{but $e$ is not a bridge in $F$}\,.
\]
We therefore have a decompositon of the collection $\{ F \subseteq E | k(F) = i \}$ into the three disjoint classes:
\begin{align}
 {\cal A}_e^i &= \{ F\subseteq E \,|\, e \not \in F,\,  k(F) = i\},\nonumber \\
{\cal B}_e^i &= \{ F\subseteq E \,|\, e \in F,\,  k(F) = i,\, k(F\setminus \{e\}) = k(F) +1 \},\label{def:Bei} \\
{\cal C}_e^i &= \{ F\subseteq E \,|\, e \in F,\,  k(F) = i,\, k(F\setminus \{e\}) = k(F) \}\,.\nonumber
\end{align}
Hence, for each $i=1,\dots,n-1$ we have
\[
  a_i = \sum_{F\in{\cal A}_e^i} (-1)^{|F|} + \sum_{F \in {\cal B}_e^i} (-1)^{|F|}  + \sum_{F \in
    {\cal C}_e^i} (-1)^{|F|}.
\]
Clearly, the mapping $\phi$ defined by $\phi(F) = F\cup \{e\}$ is a bijection from
 ${\cal A}_e^i$ to ${\cal B}^{i-1}_e \cup {\cal C}_e^i$, which implies that
\[
  \sum_{F \in {\cal A}_e^i} (-1)^{|F|} = - \left(\sum_{F \in {\cal B}_e^{i-1}} (-1)^{|F|} +
  \sum_{F\in {\cal C}_e^i}(-1)^{|F|}\right).
\]
Plugging this expression into the previous formula for $a_i$, we get
\[
  a_i = \sum_{F\in {\cal B}_e^i} (-1)^{|F|} - \sum_{F\in {\cal B}^{i-1}_e} (-1)^{|F|} = b^i_e - b_e^{i-1}
\]
as desired.
\end{proof}

\begin{lem}\label{eq:lemma-subgraph-b} For $i= 1,2,3,\dots$ we have
\begin{equation}\label{eq:lemma-subgraph}
  b^i_e(G) =  - \sum_{\substack{V = V_1 \sqcup \cdots \sqcup V_{i+1}\\  e \not\in G[V_j],\,j=1,\dots,i+1}}\prod_{j=1}^{i+1}a_1( G[V_j]) \,,
\end{equation}
where $V = V_1 \sqcup \cdots \sqcup V_{i+1}$ denotes any decomposition of $V$ into $i+1$ (non-empty) disjoint subsets $V_1,\dots, V_{i+1}$.
\end{lem}
\begin{proof}
Let $F\in {\cal B}^i_e$  and let $G_1=(V_1,F_1),\dots, G_{i+1}=(V_{i+1},F_{i+1})$ be the connected components of ${\bar G}\langle F\setminus\{e\}\rangle$. In this way, $F$ defines a decomposition of $V$ into $i+1$ disjoint sets $V_1, \dots, V_{i+1}$
such that $e\not \in G[V_j]$ for any $j=1,\dots,i+1$. Let $E_1, \dots E_{i+1}$ be the edge sets of the
vertex induced subgraphs $G[V_1], \dots, G[V_{i+1}]$, respectively. Note that $F$ decomposes as $F_1\cup \cdots\cup F_{i+1} \cup\{e\}$,
where $F_j\subseteq E_j$ for each $j$. Conversely, given a decomposition of $V$ into $i+1$ subsets as above such that no $G[V_j]$ contains $e$, then $F=F_1\cup\dots\cup F_{i+1}\cup\{e\}$ belongs to ${\cal B}^i_e$ for any collection $F_1, \dots, F_{i+1}$ of edge sets in $G[V_1],\dots, G[V_{i+1}]$, respectively, such that each $(V_j,F_j)$ is connected. Hence, we can organize the sum over $F\in B_e^i$ by aggregating
terms with the same decomposition of $V$: denoting by $k(F_j)$ the number of connected components of
$(V_j,F_j)$, we have:
\begin{multline*}
  b_e^i(G) = \sum_{F\in {\cal B}_e^i} (-1)^{|F|} =
  \sum_{\substack{V  =V_1\sqcup \cdots \sqcup V_{i+1} \\ e\not \in G[V_j],\,j=1,\dots,i+1}} \sum_{\substack{F_j\subseteq E_j,\,k(F_j)=1\\j=1,\dots,i+1}}
  (-1)^{1+ \sum_{j=1}^{i+1}|F_j| } \\ = -
  \sum_{\substack{V=V_1\sqcup \cdots\sqcup V_{i+1}\\e \not \in G[V_j],\,j=1,\dots,i+1 }}
  \prod_{j=1}^{i+1} \sum_{\substack{F_j\subseteq E_j\\k(F_j)=1}} (-1)^{|F_j|} =
  - \sum_{\substack{V=V_1\sqcup \cdots\sqcup V_{i+1}\\e \not \in G[V_j],\,j=1,\dots,i+1}} \prod_{j=1}^{i+1}a_1(G[V_j])\,.
\end{multline*}
\end{proof}
Note that only decompositions such that $G[V_j]$ is connected for all $j=1,\dots,i+1$ contribute to the right-hand side of \eqref{eq:lemma-subgraph}, since $a_1$ vanishes for disconnected graphs.

Next, we proceed to verify a similar set of recursion relations for the $h_i$.
For this purpose, assume a linear ordering of the edges of the graph $G=(V,E)$ is given and let us call a  set of edges $F\subseteq E$ an \emph{$i$-forest} if $\bar G\langle F\rangle$ has $i$ components each of which is a tree, i.e.\ $\bar G\langle F\rangle$ is an acyclic graph with $k(F)=i$. Since for each tree the number of edges is one less than the number of vertices, we have that $i=k(F)=n-|F|$ for any $i$-forest $F$. Thus every spanning
$i$-forest is a subgraph with $n-i$ edges. Conversely, since every cycle trivially contains a broken
cycle as a subset, any subgraph of $G$ which does not contain any broken cycle is an $i$-forest,
 if it has $n-i$ edges. In conclusion, $h_i(G)$ is the number
of spanning $i$-forests of $G$ containing no broken cycle.

\begin{lem}\label{lemma:forest}
  For any graph $G=(V,E)$ with a linear ordering of $E\neq \emptyset$ we have that
  \begin{equation}\label{eq:decomposition-h}
    h_i(G) = c_{i-1}(G) + c_i(G),
  \end{equation}
  where the numbers $c_i(G),\, i=1,2,3,\dots,$ are given by
  \begin{equation}\label{eq:lemma-forest}
    c_i(G) = \sum_{\substack{V = V_1 \sqcup \cdots \sqcup V_{i+1}\\ e_{\max}\not\in G[V_j],\, j=1,\dots,i+1}}
    \prod_{j=1}^{i+1} h_1(G[V_j])\,.
  \end{equation}
  and $e_{\max}$ is the maximal edge of $G$, while $c_0(G)=0$.
\end{lem}
\begin{proof}
Let $F$ be an $i$-forest of $G$ for some $i$, and fix $e\in E$. Then exactly one of the following is true:
\begin{enumerate}
\item $e\notin F$, and $F\cup\{e\}$ is not a forest (i.e.\ adding $e$ to $F$ creates a cycle),
\item $e\notin F$, and $F\cup\{e\}$ is an $(i-1)$-forest,
\item $e\in F$, and $F\setminus\{e\}$ is a $(i+1)$-forest.
\end{enumerate}
If we now choose $e=e_{\max}$ and $F$ is an $i$-forest such that case 1) holds, then $F$ has a broken cycle.
If we therefore consider forests which contain no broken cycle, case 1) does not occur and we can therefore decompose the set
\[{\cal E}^i = \{ F\subseteq E \,|\, F \text{ is a spanning $i$-forest with no broken cycle} \}\]
into two disjoint classes:
\begin{align*}
  \tilde{\cal A}^i_{e_{\max}} &= \{ F\in {\cal E}^i \,|\, e_{\max} \not \in F\}, \\
  \tilde{\cal B}^i_{e_{\max}} &= \{ F\in {\cal E}^i \,|\, e_{\max} \in F\} \\
\end{align*}
and, clearly, $F\mapsto F\cup\{e_{\max}\}$ is a bijection from $\tilde{\cal A}^i_{e_{\max}}$ onto $\tilde{\cal B}^{i-1}_{e_{\max}}$.
If we now define $c_i(G) = |\tilde{\cal B}^i_{e_{\max}}|$ and recall that $h_i(G) = |{\cal E}^i|$, we see that
\begin{equation}
  h_i(G) = c_{i-1}(G) + c_i(G)\,,\quad i=1,2,3,\dots\,.
\end{equation}
Note that $c_0(G)=0$ since ${\cal E}^0$ is empty. We have to show that the $c_i(G)$ given in \eqref{eq:lemma-forest} coincide with the ones
we have just defined.

Let $F\in \tilde{\cal B}^{i}_{e_{\max}}$. Then  $F\setminus\{e_{\max}\}$ is a spanning
$(i+1)$-forest and it can be written as the disjoint union of its components:
$$F\setminus\{e_{\max}\} = T_1\cup \cdots \cup T_{i+1}\,.$$
Let $V_j$ be the vertex set of $T_j$ and let $G_j = G[V_j]$ be the corresponding vertex induced subgraph of
$G$, for each $j=1,\dots,i+1$. Then $T_j$ is a spanning tree of $G_j$. Since $F$ contains no broken
cycle by assumption, neither does any of the $T_j$ and, in particular, $e_{\max} \not\in G_j$ for every $j$.

Conversely, consider a decomposition $V=V_1\sqcup \cdots \sqcup V_{i+1}$ such that
$e_{\max} \not\in G_j = G[V_j]$ for every $j=1,\dots, i+1$.
If $T_j$ is a spanning tree for $G_j$ for each $j$ then
$F=T_1\sqcup \cdots \sqcup T_{i+1}\sqcup \{e_{\max}\}$ is a spanning $i$-forest of $G$. If none of the
$T_j$ contains a broken cycle, then neither will $F$. This proves the formula.
\end{proof}
As in formula \eqref{eq:lemma-subgraph} only decompositions such that all $G[V_j]$ are connected contribute to the sum in \eqref{eq:lemma-forest}.

\begin{proof}[Proof of Theorem \ref{thm:whitney}]
With notation as in Lemmas \ref{lemma:subgraph} and \ref{lemma:forest} we define
\[
\tilde a_i(G) = (-1)^{n-i} h_i(G)\qquad\mbox{and}\qquad \tilde b_e^i(G)= (-1)^{n-i}c_i(G)
\]
for $i=1,2,\dots, n$ and $i= 0,1,\dots, n$, respectively, (where $e=e_{\rm max}$). It follows from \eqref{eq:decomposition-h} and \eqref{eq:lemma-forest}
that $\tilde a_i$ and $\tilde b^i_e$ satisfy the same recursion relations \eqref{eq:decomposition-a} and \eqref{eq:lemma-subgraph} as $a_i$ and $b^i_e$. Specialising \eqref{eq:lemma-subgraph} to $i=1$ and noting that $a_1=b_e^1$ we get
  \begin{equation}\label{eq:proof-induction-1}
    a_1(G) = - \sum_{\substack{V=V_1\sqcup V_2\\e \not \in G[V_j],\, j=1,2 }}  a_1(G[V_1])\cdot a_1(G[V_2])\,.
  \end{equation}
	Noting that for the case of the empty graph $\bar G\langle \emptyset \rangle$ it holds that
	\[
	a_1(\bar G\langle \emptyset \rangle) = \begin{cases} 1,\; \mbox{if $n=1$}\\ 0,\, \mbox{if $n>1$}\end{cases}
	\]
	this relation determines $a_1(G)$ uniquely for all graphs $G$ by induction, since the graphs $G[V_j]$ have fewer edges than $G$. In turn, relations \eqref{eq:decomposition-a} and \eqref{eq:lemma-subgraph} determine  $a_i(G)$ for $i\geq 2$.

	Since it is clear that $a_1(\bar G\langle \emptyset \rangle)=\tilde a_1(\bar G\langle \emptyset \rangle)$ and $\tilde a_1(G)=\tilde b_e^1(G)$ it follows that $a_i(G)=\tilde a_i(G)$ for all $i$ and all graphs $G$.
\end{proof}

It is a well known fact that the coefficients $a_i(G)$ alternate in sign, and that they are
numerically upper bounded by the corresponding coefficients for the complete graph of equal
order. We will now briefly show how this follows in a simple manner from the recursion relations of Lemmas~\ref{lemma:subgraph} and \ref{eq:lemma-subgraph-b} without using neither Whitney's theorem nor the deletion-contraction principle, as a consequence of the following result.

\begin{lem}\label{lem:pos-mon}
  For any graph $G$ of order $n$ and any edge $e$ of $G$ it holds that
  \begin{equation}\label{pos-mon}
		0\le (-1)^{n-i}b_e^i(G) \le (-1)^{n-i}b_e^i(K_n)\,, \quad  i=1,\dots,n ,
		\end{equation}
    where $K_n$ denotes the complete graph on $n$ vertices.
    Moreover, the first inequality is sharp if and only if $k(G)\leq i\leq n$, while the second inequality is sharp for $1\leq i \leq n-1$ unless $G=K_n$.
\end{lem}
\begin{proof}
  We shall prove the statement by induction. Consider first the case $i=1$ and
note that the recursion relation \eqref{eq:proof-induction-1} can be rewritten as
 \begin{equation}\label{eq:proof-induction-11}
    d(G) =  \sum_{\substack{V=V_1\sqcup V_2\\e \not \in G[V_j],\, j=1,2 }}  d(G[V_1])\cdot d(G[V_2])\,,
  \end{equation}
where $d(G)= (-1)^{n-1} a_1(G)$. Since
\[
d(V,\emptyset) = \begin{cases} 1\quad \mbox{if $|V|=1$}\\ 0\quad \mbox{if $|V|>1$}\,,\end{cases}
\]
 it follows by induction on the number of edges in $G$ that $d(G)\geq 0$ for all $G$. If $G$ is connected it is easy to see, by successively deleting edges in paths connecting the endpoints of $e$, starting with $e$, that there exist decompositions $V=V_1\sqcup V_2$ such that $G[V_1]$ and $G[V_2]$ are both connected and do not contain $e$. This implies, again by induction, that $d(G)>0$ if $G$ is connected. On the other hand, if $G$ is disconnected, the sum in \eqref{eq:proof-induction-11} is empty and so $d(G)=0$.

Using \eqref{eq:lemma-subgraph} in the form
\begin{equation}\label{eq:lemma-subgraph-11}
 (-1)^{n-i} b^i_e(G) =   \sum_{\substack{V = V_1 \sqcup \cdots \sqcup V_{i+1}\\  e \not\in G[V_j],\,j=1,\dots,i+1}}\prod_{j=1}^{i+1}d(G[V_j]) \,,
\end{equation}
we get that $(-1)^{n-i} b^i_e(G)\geq 0$. Moreover, if $G$ has $k$ connected components, the sum on the right-hand side is empty if $i<k$ whereas positive terms occur for $k\leq i\leq n$ and hence $(-1)^{n-i} b^i_e(G)> 0$ in this case.

   Moreover, considering $G$ as a subgraph of $K_n$ and comparing the formula \eqref{eq:lemma-subgraph-11} for $G$ and the corresponding one for $K_n$, we see that each summand in the former by the induction hypothesis can be bounded from above by a corresponding term in the latter, since all $K_n[V_j]$ are complete graphs. Hence, the rightmost bound in \eqref{pos-mon} follows.

	Finally, if $G$ is not the complete graph, we have $n\geq 2$ and there is an edge $f=\{x,y\}$ in $K_n$ that is not an edge of $G$. For $i\leq n-1$ we choose a decomposition $V=V_1\sqcup\dots\sqcup V_{i+1}$ in \eqref{eq:lemma-subgraph-11}, such that $V_1=\{x,y\}$ and $V_2=\{z\}$, where $z$ is an endpoint of $f$ that is not in $V_1$, and $V_3,\dots,V_{i+1}$ are arbitrary. Then $G[V_1]$ is disconnected and therefore this term in \eqref{eq:lemma-subgraph-11} vanishes, while the corresponding term for $K_n$ is strictly positive. This proves the last statement of the proposition.
\end{proof}
\begin{cly}
  For any graph $G$ with $n$ vertices it holds for $i=1,2,\dots,n$ that
 \begin{equation}\label{pos-mon1}
 0\leq (-1)^{n-i}(a_1(G) + a_2(G)+\dots + a_i(G))\leq (-1)^{n-i}(a_1(K_n) + a_2(K_n)+\dots + a_i(K_n))\,,
 \end{equation}
and
 \begin{equation}\label{pos-mon2}
 0\leq (-1)^{n-i} a_i(G)\leq (-1)^{n-i} a_i(K_n).
 \end{equation}
   Moreover, in both cases the first inequality is sharp if and only if $k(G)\leq i\leq n$, while the second inequality is sharp for $1\leq i \leq n-1$ unless $G=K_n$.
\end{cly}
\begin{proof}
Using that
\begin{equation}\label{a-b-relation}
 a_i(G) = b_e^i(G) - b_e^{i-1}(G)
\end{equation}
by \eqref{eq:decomposition-a} and that $b_e^0(G)=0$ it follows that
$$
b_e^i(G) = a_1(G) +\dots + a_i(G)\,.
$$
In particular, $b_e^i(G)$ is independent of $e$ and \eqref{pos-mon1} is just a rewriting of \eqref{pos-mon}. Writing \eqref{a-b-relation} as
$$
(-1)^{n-i}a_i(G) = (-1)^{n-i}b_e^i(G) +  (-1)^{n-i+1}b_e^{i-1}(G)
$$
the inequalities \eqref{pos-mon2} follow immediately from \eqref{pos-mon}. Moreover, the first inequality of \eqref{pos-mon2} is an equality if and only if $b_e^i(G)=b_e^{i-1}(G)=0$ and hence if and only if $0\leq i < k(G)$. Similarly, Lemma \ref{pos-mon} gives that if the second inequality of \eqref{pos-mon2} is an equality then $b_e^i(G)=b_e^{i}(K_n)$ and $b_e^{i-1}(G)=b_e^{i-1}(K_n)$ and hence $G=K_n$. This completes the proof of the corollary.
\end{proof}

It should be noted that the inequality \eqref{pos-mon1} can also easily be deduced from the (highly
non-trivial) unimodularity of the coefficients of $\chi_G$ \cite{read,huh} and the fact that
$a_1+a_2 + \dots + a_n =0$.

\begin{rem}
The alternating sign property of the $a_i$ plays a role, for the special case $i=1$, in the Mayer expansion for the hard-core lattice gas in statistical mechanics (also known as the cluster expansion of the polymer partition function)\cite{ueltschi2004cluster, Scott2005,  friedli_velenik_2017}. Briefly, the model is defined by a finite set $\Gamma$ which plays the role of the ``single-particle'' state space, a list of complex weights $w = (w_\gamma)_{\gamma \in \Gamma}$, and an interaction $W:\Gamma\times \Gamma \to \{0,1\}$, which is symmetric and satisfies $W(\gamma,\gamma)=0$ for all $\gamma \in \Gamma$.
Given a multiset $X=\{\gamma_1,\dots, \gamma_n\}$ of elements of $\Gamma$ (where each $\gamma_i$ can appear more than once), we define the simple graph $G[X] \subseteq K_n$ as the graph on $n$ vertices such that $i$ is adjacent to $j$ if $i\neq j$ and $W(\gamma_i,\gamma_j) = 0$.
A subset $X$ of $\Gamma$ is said to be \emph{independent} if $G[X]$ has no edges. The partition function is then given by
\begin{equation}
    Z_\Gamma(w) = \sum_{X\subseteq \Gamma} \Big(\prod_{\gamma \in X} w_\gamma \Big) \prod_{\{\gamma, \gamma'\} \subseteq X} W(\gamma, \gamma') =
    \sum_{\substack{X\subseteq \Gamma \\ X \text{ independent}}} \prod_{\gamma\in X} w_\gamma,
\end{equation}
which is the (generalized) independent-set polynomial of $G[\Gamma]$ (the standard independent-set polynomial is given when $w$ is taken to be constant)\cite{Scott2005}. The Mayer expansion gives a formal series expansion for $\log Z_\Gamma$ \cite[Proposition 5.3]{friedli_velenik_2017}:
\begin{equation}\label{eq:cluster-expansion}
    \log Z_\Gamma(w) = \sum_{n\ge 1} \frac{1}{n!} \sum_{\gamma_1,\dots,\gamma_n \in \Gamma}
     a_1(G[\gamma_1,\dots,\gamma_n]) \prod_{i=1}^n w_{\gamma_i} .
\end{equation}
The alternating sign property of $a_1$ implies in particular that the coefficient of order $n$ of $\log Z_G(w)$, seen as a polynomial in the variables $( w_\gamma)_{\gamma \in \Gamma}$, has sign $(-1)^{n-1}$. This holds in greater generality \cite[Proposition 2.8]{Scott2005}, and has important implications for proving the convergence of the formal series \eqref{eq:cluster-expansion} \cite{Pfister1991}.
\end{rem}

\section{The recursion relations for hypergraphs}
\label{sec:hypergraphs}

Let $H$ be a \emph{hypergraph}, that is $H=(V,E)$ where $V$ is a finite non-empty set of \emph{vertices} and $E$ is a set of subsets of $V$, called \emph{edges}. We assume all edges have cardinality at least $2$ (i.\,e. $H$ has no loops) and will denote $|V|$ by $n$.

A hypergraph $H'=(V',E')$ is a \emph{subgraph} of $H$ if $V'\subseteq V$ and $E'\subseteq E$. If
$$
E'= \{ e\in E\mid e\subseteq V'\}
$$
we call $H'$ the subgraph spanned by $V'$ and denote it by $H[V']$. If
$$
V' = \bigcup_{e\in E'} e
$$
we call $H'$ the subgraph spanned by $E'$ and denote it by $H\langle E'\rangle$. Finally, in case $V=V'$ we call $H'$ a \emph{spanning subgraph} of $H$ and denote it by $\bar H\langle E'\rangle$.

Two different vertices $x,y\in V$ are called \emph{neighbours} in $H$ if $x,y\in e$ for some $e\in
E$. A vertex $x$ is \emph{connected} to a vertex $y$ if either $x=y$ or there exists a finite sequence $x_1,x_2,\dots x_k$ of vertices such that $x_i$ and $x_{i+1}$ are neighbours for $i=1,\dots,k-1$ and $x_1=x$ and $x_k=y$. Clearly, connectedness is an equivalence relation on $V$. Calling the equivalence classes $V_1,\dots, V_N$ and letting $E_i$ be the set of edges containing only vertices of $V_i$, we have that $H_i=(V_i,E_i)$ is a hypergraph and
$$
V=\bigcup_{i=1}^N V_i \,,\qquad E=\bigcup_{i=1}^N E_i\,.
$$
If $N=1$ we call $H$ \emph{connected}. Evidently, $H_1,\dots,H_N$ are connected. They are called the \emph{connected components} of $H$ and their number is denoted by $k(H)$. Again, we shall use the notation $k(F)$ for $k(\bar H\langle F\rangle)$.

\begin{defn}
Let $\lambda\in\mathbb N$. A $\lambda$-colouring of a hypergraph $H=(V,E)$ is a map $\pi: V\to \{1,2,\dots,\lambda\}$. A $\lambda$-colouring is called proper if for each edge $e\in E$ there exist vertices $x,y\in e$ such that $\pi(x)\neq \pi(y)$. We define
$\chi_H(\lambda)$ to be the number of proper $\lambda$-colourings of $H$.
\end{defn}

Repeating the proof of Theorem \ref{chrom1} we obtain

\begin{thm}\label{chrom2}
The function $\chrom{H}$ is a polynomial, called the \emph{chromatic polynomial} of $H$, given by
$$
\chi_H(\lambda) = \sum_{F\subseteq E} (-1)^{|F|} \lambda^{k(F)}\,.
$$
\end{thm}

Thus, the coefficients $a_i(H),\, i=1,2,3,\dots,n,$ of $\chi_H$ are given by the same formula \eqref{coeff} as for graphs.

Now, fix $e\in E$ and let
\begin{align*}
{\cal A}_e^i &= \{F\subseteq E\mid e\notin F, k(F)=i\}\\
{\cal B}_e^{i,j} &= \{F\subseteq E\mid e\in F, k(F) =i, k(F\setminus\{e\}) = j\}\,.
\end{align*}
Note that ${\cal B}_e^{i,j}=\emptyset$ if $i>j$ and, if $F\in{\cal A}_e^i $, then $F\cup \{e\} \in{\cal B}_e^{j,i}$ for some $j\leq i$ yielding a bijective correspondence between ${\cal A}_e^{i}$ and $\cup_{j\leq i}{\cal B}_e^{j,i}$. Hence, we have
\begin{equation}
\sum_{F\in {\cal A}_e^i} (-1)^{|F|} = - \sum_{j=1}^i \sum_{F\in{\cal B}_e^{j,i}} (-1)^{|F|}\,.
\end{equation}
Using
\begin{equation}
 a_i = \sum_{F\in {\cal A}_e^i} (-1)^{|F|} + \sum_{j=i}^n \sum_{F\in{\cal B}_e^{i,j}} (-1)^{|F|}
\end{equation}
it follows that
 \begin{equation}\label{I}
a_i = \sum_{j>i} b_e^{i,j} - \sum_{j<i} b_e^{j,i}\,,
\end{equation}
where
\begin{equation}\label{defb}
b_e^{i,j} = \sum_{F\in {\cal B}_e^{i,j}} (-1)^{|F|}\,.
\end{equation}
In particular, we have
\begin{equation}\label{I'}
a_1 = \sum_{j=2}^n b_e^{1,j}\,.
\end{equation}

\begin{prop}\label{bformula}
For $i<j$ it holds that
\begin{equation}\label{II}
b_e^{i,j} =  - \sidesum{i}_{V_1\sqcup\dots\sqcup V_j=V}\prod_{k=1}^j a_1(H[V_k])\,,
\end{equation}
where the sum is over all decompositions of $V$ into j (non-empty) disjoint subsets such that $e$ intersects exactly $j-i+1$ of them.
\end{prop}
\begin{proof}
Let $F\in {\cal B}_e^{i,j}$. Then $\bar H\langle F\rangle$ has $i$ components $C_1,\dots,C_i$, whereas $\bar H\langle F\setminus\{e\}\rangle$ has $j$ components $H_1=(V_1,F_1),\dots, H_j=(V_j,F_j)$ which are connected spanning subgraphs of $H[V_1],\dots, H[V_j]$, respectively.  Indeed, we have $e\in C_m \equiv (V',F')$ for some $m=1,\dots,i$, and $(V',F'\setminus\{e\})$ then has $j-i+1$ components which together with $\{C_1,\dots,C_{m-1},C_{m+1},\dots,C_i\}$ make up $\{H_1,\dots,H_j\}$, and $e$ intersects exactly those $V_k$ which originate from $C_m$ by deleting $e$.

On the other hand, given a decomposition $V_1\sqcup\dots\sqcup V_j$ of $V$ and connected spanning
subgraphs $H_1=(V_1,F_1),\dots,H_j=(V_j,F_j)$ of $H[V_1],\dots, H[V_j]$, respectively, such that $e$
intersects exactly $j-i+1$ of $V_1,\dots, V_j$, we get that $F_1\cup\dots\cup F_j\cup\{e\}\in{\cal
  B}_e^{i,j}$ and the mapping $\psi$ defined by
$$
\psi(\{H_1,\dots, H_j\}) = F_1\cup\dots\cup F_j\cup\{e\}
$$
is a bijection  onto ${\cal B}_e^{i,j}$.

Since
$$
(-1)^{|F_1\cup\dots\cup F_j\cup\{e\}|} = - \prod_{k=1}^j (-1)^{|F_k|}\,,
$$
the claim follows upon noting that $a_1(H[V'])=0$ if $H[V']$ is not connected.
\end{proof}

Setting $i=1$ and summing over $j$ in \eqref{II} we get
\begin{equation}\label{II'}
a_1(H) = - \sum_{j=2}^n \sidesum{1}_{V_1\sqcup\dots\sqcup V_j=V}\prod_{k=1}^j a_1(H[V_k])\
\end{equation}
which determines $a_1(H)$ inductively for any hypergraph $H$, since $H[V_1],\dots, H[V_j]$ all have fewer edges than $H$ and we obviously have
\begin{equation}\label{III}
a_1(\bar{H}\langle \emptyset \rangle) = \begin{cases} 1\; \mbox{if $|V| = 1$} \\ 0\; \mbox{if $|V|>1$}\,.\end{cases}
\end{equation}

Once $a_1$ is known we obtain $b_e^{i,j}(H)$ for any $H$ from \eqref{I'} and consequently $a_i(H)$ from \eqref{I}. Hence, equations \eqref{I}, \eqref{II} and \eqref{III} determine all $a_i$ (as well as all $b_e^{i,j}$).

We will now present a generalization of Whitney's broken cycle theorem for hypergraphs.
\begin{defn}\label{delta}
Let $H=(V,E)$ be a hypergraph and fix some linear ordering $\leq$ of $E$. A non-empty set $F\subseteq E$ is called \emph{broken-cyclic} in $H$ with respect to $\leq$ if it fulfils the following property

\smallskip

($\bigstar$)\quad $H\langle F\rangle$ is connected and there exists an edge $e_0\subseteq \bigcup_{f\in F} f$ such that $e_0 > \mbox{max}\, F$.
\end{defn}

\begin{lem}\label{stardec}
Assume $H=(V,E)$ is a hypergraph with connected components $H_1(V_1,E_1),$ \dots, $H_N=(V_N,E_N)$. Then $F\subseteq E$ is broken-cyclic in $H$ if and only if $F\subseteq E_i$ and $F$ is broken-cyclic in $H_i$ for some $i=1,\dots,N$, with ordering of edges inherited from that of $H$.
\end{lem}

\begin{proof}
If $F$ is broken-cyclic in $H$ then $H\langle F\rangle$ is connected and hence is a subgraph of some $H_i$. Consequently, if $e_0\subseteq \bigcup_{f\in F} f$  it is an edge of $H_i$ and it follows that $F$ is broken-cyclic in $H_i$.

The converse, that a set of edges $F$ which is broken-cyclic in $H_i$ is also broken-cyclic in $H$, is obvious.
\end{proof}

From now on $H=(V,E)$ is a fixed hypergraph with some linear ordering $\leq$ on $E$ and $\cal D$ is some subsetset of $2^E$ consisting of broken-cyclic subsets in $H$ with respect to $\leq$. Moreover, if $H'=(V',E')$ is a subgraph of $H$ it will be assumed that $E'$ is ordered with respect to the restriction of $\leq$ to $E'$.

We define
\begin{align}
{\cal E}_{\cal D} &= \{F\subseteq E\mid A\nsubseteq F \;\mbox{for all $A\in \cal D$}\} \\
{\cal E}^i_{\cal D} &= \{F\subseteq E\mid k(H\langle F\rangle) = i \}\cap {\cal E}_{\cal D}
\end{align}
and set
\begin{equation}
a_{i,\cal D} = \sum_{F\in{\cal E}^i_{\cal D}} (-1)^{|F|}\,,
\end{equation}
for $i=1,2,3,\dots,n$. Note that $a_i = a_{i,\emptyset}$.

We may now formulate the following version of the broken-cycle theorem.

\begin{thm}\label{thyp}
For any set $\cal D$ of broken-cyclic subsets of edges in a hypergraph $H$ it holds that
\begin{equation}
a_i = a_{i,\cal D}
\end{equation}
for all $i$.
\end{thm}

\begin{proof} Let $e=\mbox{max}\, E$. Defining the sets
\begin{equation}
{\cal A}^i_{e,\cal D} = {\cal A}^i_{e}\cap {\cal E}_{\cal D}\,,\qquad  {\cal B}^{i,j}_{e,\cal D} = {\cal B}^{i,j}_{e}\cap{\cal E}_{\cal D}\,,
\end{equation}
we have the decomposition
\begin{equation}
{\cal E}^i_{\cal D} = {\cal A}^i_{e,\cal D}\cup\left(\bigcup_{j\geq i}{\cal B}^{i,j}_{e,\cal D}\right)
\end{equation}
into disjoint subsets. Moreover, since $e$ is maximal in $E$ it does not belong to any broken-cyclic
subset in $H$ and therefore the mapping $\varphi$ defined by $\varphi(F) = F\cup\{e\}$ is a bijection from ${\cal A}^i_{e,\cal D}$ onto $\bigcup_{j\leq i} {\cal B}^{j,i}_{e,\cal D}$. Thus, defining
\begin{equation}
b^{i,j}_{e,\cal D} = \sum_{F\in{\cal B}^{i,j}_{e,\cal D}} (-1)^{|F|}\,,
\end{equation}
the same arguments as those leading to relation \eqref{I} imply
\begin{equation}\label{Ihyp}
a_{i,\cal D} = \sum_{j>i} b_{e,\cal D}^{i,j} - \sum_{j<i} b_{e,\cal D}^{j,i}\,.
\end{equation}

We next argue that the analogue of \eqref{II} also holds. Let $F\in {\cal B}^{i,j}_{e,\cal D}$ and consider the corresponding connected components $H_1=(V_1,F_1),\dots, H_N= (V_j,F_j)$ of the subgraph $\bar H\langle F\setminus\{e\}\rangle$ (see the proof of Proposition~\ref{bformula}). For $A\in \cal D$ we have by Lemma~\ref{stardec} that $A\subseteq F$ if and only if $A\subseteq F_k$ for some $k=1,\dots, j$. Defining
\begin{equation}
{\cal D}_k = {\cal D}\cap 2^{E_k}\,,
\end{equation}
where $E_k$ denotes the edgeset of $H[V_k]$, this means that $A\nsubseteq F$ for all $A\in{\cal D}$ if and only if $A\nsubseteq F_k$ for all $A\in{\cal D}_k$ and all $k=1,\dots, j$. Observe that any $A\in{\cal D}_k$ is broken-cyclic in $H[V_k]$ since the vertices of edges in $A$ belong to $V_k$ and hence $H\langle A\rangle = H[V_k]\langle A\rangle$. We conclude that $F\in {\cal B}^{i,j}_{e,\cal D}$ if and only if $F_k\in {\cal A}^{1}_{e,{\cal D}_k}(H[V_k])$ for all $k=1,\dots, j$.

As in the proof of Proposition~\ref{bformula} we obtain, conversely, from any decomposition $V_1\sqcup\dots\sqcup V_j=V$ and connected, spanning subgraphs $H_1=(V_1,F_1),\dots, H_j= (V_j,F_j)$ of $H_1=H[V_1],\dots, H_j= H[V_j]$ such that $A\nsubseteq F_k$ for all $A\in {\cal D}_k$ and all $k=1,\dots, j$, and such that $e$ intersects exactly $j-i+1$ of the sets $V_1,\dots, V_j$, that $F=F_1\cup\dots F_j\cup\{e\}$ belongs to ${\cal B}^{i,j}_{e,\cal D}$. Hence we obtain the desired relation
\begin{equation}\label{IIhyp}
b^{i,j}_{e,\cal D}(H) = -\sidesum{i}_{V_1\sqcup\dots\sqcup V_j=V} \prod_{k=1}^j a_{1,{\cal D}_k}(H[V_k])\,,
\end{equation}
where one should note that ${\cal D}_k$ depends solely on $V_k$ and $\cal D$ for a given $H$.

Having established equations \eqref{Ihyp} and \eqref{IIhyp} the claimed equality of $a_i$ and $a_{i,{\cal D}}$ follows by induction on the number of edges since, if $E=\emptyset$, we must have $\cal D=\emptyset$ and so
\begin{equation}
  a_{i,\cal D}(\bar H\langle \emptyset \rangle) = a_i(\bar H\langle \emptyset \rangle)\,,\quad i=1,2,3,\dots
	\end{equation}
\end{proof}

The following Propositions \ref{deltacond} and \ref{cyclecond} show that Theorem~\ref{thyp} contains the broken cycle theorems of \cite{dohmen1, dohmen2, trinks} and those quoted for hypergraphs in \cite{dohmen-trinks}.

\begin{prop}\label{deltacond}
Assume $H'=(V',F)$ is a $\delta$-cycle in $H=(V,E)$ in the sense of \cite{trinks}, i.e.\ $H'$ is a minimal subgraph of $H$ such that $F\neq \emptyset$ and $k(H')=k(H'-e)$ for all $e\in F$. Then $F\setminus\{\mbox{\rm max}\, F\}$ is broken-cyclic in $H$ according to Definition~\ref{delta}.
\end{prop}

\begin{proof}
Since $H'$ is minimal it follows that $k(H')=k(H'-e)=1$ for all $e\in F$. In particular, $H'- \mbox{max}\, F$ is connected and equals $H\langle F\setminus\{\mbox{max}\, F\}\rangle$ with vertex set $V'$. Hence, $\mbox{max}\, F\subseteq \bigcup_{f\in F\setminus\{\mbox{max}\, F\}} f$ and, of course, $\mbox{max}\, F> \mbox{max}(F\setminus\{\mbox{max}\, F\})$.
\end{proof}

\begin{prop}\label{cyclecond}
Let $C=x_1e_1x_2e_2 \dots x_ne_nx_1$ be a cycle in $H$ in the sense of \cite{berge}, i.e.\ $x_1$, \dots, $x_n$, resp. $e_1,$ \dots, $e_n$, are pairwise distinct vertices, resp. edges, in $H$ such that $x_i\in e_{i-1}\cap e_i$ for $i=1,\dots, n$ (with $e_0\equiv e_n$). Setting $F=\{e_1,\dots, e_n\}$ we have that $F\setminus\{\mbox{\rm max}\, F\}$ is broken-cyclic in $H$ provided
\begin{equation}\label{incl}
\mbox{\rm max}\, F\subseteq \bigcup_{f\in F\setminus\{\mbox{\small\rm max}\, F\}} f\,,
\end{equation}
which in particular holds if $\mbox{\rm max}\, F$ has cardinaliy 2.
\end{prop}

\begin{proof}
It is clear that $H\langle F\setminus\{\mbox{max}\, F\}\rangle$ is connected and that \eqref{incl} ensures that we may use $e_0 = \mbox{max}\, F$ in Definition~\ref{delta}.

If $\mbox{max}\, F =e_k$ has cardinality $2$ then $e_k=\{x_k,x_{k+1}\}\subseteq e_{k-1}\cup e_{k+1}\subseteq\cup_{ f\in F\setminus\{e_k\}} f$.
\end{proof}

Alternating sign properties of the $a_i$ for hypergraphs such as the ones described in Section
\ref{sec:graphs} for graphs have been demonstrated in some specific cases, see
e.g. \cite{dohmen1}. To what extent analogues of \eqref{pos-mon2} can be obtained in the general
case of hypergraphs is not clear. We should mention on this topic that the deletion-contraction
principle has been extended to hypergraphs \cite{Zykov_1974} as well as to mixed hypergraphs
\cite{voloshin1993}.

\section{An application: the first chromatic coefficient for complete hypergraphs}
\label{sec:complete-hypergraphs}
As a last topic we show that the recursion relations of Section~\ref{sec:hypergraphs} can be used to derive the value of $a_1$ for complete hypergraphs. Let $K_n^r$ be the $r$-complete hypergraph of order $n$, i.e.\ the edge set of $K_n^r$ consists of all
$r$-subsets of its vertex set $V=\{1,2,\dots,n\}$.
Note that if $r=2$, then $K^2_n$ is the complete graph $K_n$ and the result is well known (see e.g. \cite{Dong2005}).

We shall calculate $a_1(K_n^r)$ for $r\geq 2$ and $n\geq 1$ making use of \eqref{II'}, which in this case takes the form
\begin{align}\label{recompl}
a_1(K_n^r) = - \sum_{j=2}^r \sum_{\substack{1\leq k_1\leq\dots\leq k_j\\ k_1+\dots +k_j=r}} N^r_{k_1,\dots,k_j}\!\sum_{\substack{s_1,\dots, s_j\geq 0\\ s_1+\dots +s_j=n-r}}\! \binom{n-r}{s_1 \dots s_j} \cdot a_1(K^r_{k_1+s_1})\cdot\ldots \cdot a_1(K_{n_j+s_j}^r)\,,
\end{align}
where $N^r_{k_1,\dots,k_j}$ denotes the number of partitions of $\{1,\dots,r\}$ into $j$ sets of size $k_1,\dots, k_j$ and $\binom{n-r}{s_1 \dots s_j}$ is the standard multinomial coefficient.

Note also that we obviously have
\begin{equation}
\chi_{\small K_n^r}(\lambda) = \begin{cases} \lambda^n\quad \mbox{if $0\leq n<r$}\\
 \lambda^n -\lambda \quad \mbox{if $n=r$}\end{cases}\,,
\end{equation}
so that, in particular,
\begin{equation}\label{inicond}
a_1(K_n^r) =\begin{cases} 1\quad \mbox{if $n=1$}\\ 0\quad \mbox{if $n=2,3,\dots r-1$}\end{cases}
\end{equation}
(while $a_1(K_n^n) = -1$).

\begin{thm}\label{thm3} For $r\geq 2$ and $n\geq 1$ it holds that
\begin{equation}\label{eq:a1-formula}
a_1(K_n^r) = - (n-1)!\, \mu_{r-1}(n)\,,
\end{equation}
where
\begin{equation}
    \mu_{r}(n) = \sum_{i=1}^{r} R_i^{-n}
\end{equation}
and $R_1,\dots, R_{r}$ denote the roots of the $r$'th Taylor polynomial $E_{r}$ of $\exp$.
\end{thm}

\begin{proof} Fix $r\ge 2$. We introduce the generating function $g(x)$ given by
\begin{equation}\label{genfct}
g(x) = \sum_{n=0}^\infty \frac{a_1(K^r_{n+1})}{n!} x^n
\end{equation}
and rewrite equations \eqref{recompl}-\eqref{inicond} as
\begin{equation}\label{eq:diffgen}
g^{(r-1)}(x) = - \sum_{j=2}^r \sum_{\substack{1\leq k_1\leq\dots\leq k_j\\k_1+\dots +k_j =r}} N^r_{k_1, \dots, k_j} g^{(k_1-1)}(x)\cdot\ldots\cdot g^{(k_j-1)}(x)
\end{equation}
with initial condition
\[
g(0)=1\,,\quad g'(0)=g''(0) =\dots = g^{(r-2)}(0) =0\,.
\]
Given two $C^\infty$-functions $\psi$ and $\varphi$ of a real variable we recall the formula
\begin{equation}\label{compleib}
(\psi\circ\varphi)^{(r)}(x) = \sum_{j=1}^r \sum_{\substack{1\leq k_1\leq\dots\leq k_j\\k_1+\dots k_j =r}} N^r_{k_1,\dots, k_j} \psi^{(j)}(\varphi(x)) \varphi^{(k_1)}(x)\cdot\dots\cdot \varphi^{(k_j)}(x)\,,
\end{equation}
which is easy to verify by induction. For $\psi=\exp$ this gives
$$
\exp(-\varphi(x))\left(\exp\circ\varphi\right)^{(r)}(x) = \sum_{j=1}^r \sum_{\substack{1\leq k_1\leq\dots\leq k_j\\k_1+\dots k_j =r}} N^r_{k_1,\dots, k_j} \varphi^{(k_1)}(x)\cdot\dots\cdot \varphi^{(k_j)}(x)\,.
$$

Setting $g=\varphi'$ in \eqref{eq:diffgen} and using $N^{(r)}_{1,1,\dots 1}=1$ it follows that $\varphi$ satisfies
$$
(\exp\circ\varphi)^{(r)} (x) = 0\,,
$$
and hence that $\exp\circ\varphi$ equals a polynomial $P$ of degree at most $r-1$. Thus
$$
g(x) = \frac{P'(x)}{P(x)}\,.
$$
The initial conditions are easily seen to imply that $P=E_{r-1}$ and consequently
$$
g(x) = \frac{E'_{r-1}(x)}{E_{r-1}(x)}
= \sum_{i=1}^{r-1} \frac{1}{x-R_i}\,,
$$
which gives the claimed result.
\end{proof}

\begin{rem}
For $r=2$ we have $R_1=-1$ and we get from Theorem \ref{thm3} the known result
\begin{equation}
a_1(K_n) = a_1(K^2_n) = (-1)^{n-1} (n-1)!\,.
\end{equation}
By inserting this value into \eqref{eq:lemma-subgraph}, we obtain an expression for $a_i(K_n)$ for all $i$. It should be noted though that the value of $a_i(K_n)$ is equal to $ s(n,i)$, where $s(n,i)$ denotes the signed Stirling numbers of the first kind.
\end{rem}

\begin{rem}
For $r=3$ the roots of $E_2$ are $R_{\pm} = -1\pm i$ which gives
\begin{equation}
a_1(K_n^3) = (-1)^{n-1}(n-1)!\, 2^{1-\frac n2}\cos\frac{n\pi}{4}
\end{equation}
\end{rem}

For the calculation of $a_1(K^r_n)$ for larger values of $r$ one may use the results available in the literature for the moment function $\mu_r(n)$. In particular, the value of $\mu_r(n)$ was computed for $n\le 2(r+1)$ \cite[Theorem 7]{Zemyan2005}, which gives the following expression for $a_1(K^r_n)$, expanding the one given in \eqref{inicond}
\begin{equation}
a_1(K_n^r) = \begin{cases}
    1 & \text{if $n=1$}\\
    0 & \text{if $2\le n \le r-1$}\\
    (-1)^{n-r+1}\binom{n-1}{r-1} & \text{if $r \le n \le 2r-1$}\\
    - [1+(-1)^{r}]\binom{2r-1}{r}& \text{if $n=2r$}
\end{cases}.
\end{equation}
In \cite{Zemyan2005} it was also shown that, once $\mu_r(n)$ is known for $r$ consecutive values of $n$, then it is possible to recursively determine the value of $\mu_r(n)$ for every $n$. This recursive formula for $\mu_r(n)$, when expressed in terms of $a_1(K^r_n)$, reads as:
\begin{equation}
\sum_{j=0}^{r-1} \binom{r-2+m}{r-1-j} a_1(K^{r}_{j+m}) = 0, \quad \forall m \in \N.
\end{equation}

On a more general note, the properties of the zeros of the Taylor polynomials of $\exp$ have been intensively investigated, starting from the work of Szeg{\"o} \cite{szego1924eigenschaft} and Dieudonn\'e \cite{dieudonne1935zeros}, who showed that the points $\frac{R_i}{r}$ accumulate, as $r$ goes to infinity, on a closed curve contained in the unit circle, now known as the Szeg{\"o} curve. See also \cite{Newman1972,Newman1976,Buckholtz1966,Conrey1988,Pritsker1997,Walker2003,varga2008dynamical,Zemyan2005} for further developements.
\bigskip

\noindent{\bf Acknowledgement} The authors acknowledge support from the Villum Foundation via the QMATH Centre of Excellence (Grant no. 19959). A.~L. acknowledges support from the Walter  Burke  Institute  for Theoretical Physics in the form of the Sherman Fairchild Fellowship as well as support from the Institute for Quantum  Information  and  Matter  (IQIM),  an  NSF  Physics Frontiers Center (NFS Grant PHY-1733907).

\bibliography{bibliography}

\begin{bibdiv}
\begin{biblist}

\bib{berge}{book}{
      author={Berge, C.},
       title={Hypergraphs},
    subtitle={Combinatorics of {F}inite {S}ets},
   publisher={North Holland},
     address={Amsterdam},
        date={1989},
        ISBN={9780444874894},
}

\bib{birkhoff}{article}{
      author={Birkhoff, G.D.},
       title={A {D}eterminant {F}ormula for the {N}umber of {W}ays of
  {C}oloring a {M}ap},
        date={1912},
     journal={The Annals of Mathematics},
      volume={14},
       pages={42\ndash 46},
}

\bib{blass-sagan}{article}{
      author={Blass, A.},
      author={Sagan, B.E.},
       title={Bijective proofs of two broken circuit theorems},
        date={1986},
     journal={Journal of Graph Theory},
      volume={10},
      number={1},
       pages={15\ndash 21},
}

\bib{Buckholtz1966}{article}{
      author={Buckholtz, J.D.},
       title={A {C}haracterization of the {E}xponential {S}eries},
        date={1966-04},
     journal={The American Mathematical Monthly},
      volume={73},
      number={4},
       pages={121},
         url={https://doi.org/10.2307/2313761},
}

\bib{bujitas2015}{incollection}{
      author={Bujt\'as, Cs.},
      author={Tuza, Zs.},
      author={Voloshin, V.~I.},
       title={Hypergraph colouring},
        date={2015},
   booktitle={Topics in chromatic graph theory},
      editor={Beineke, L.W.},
      editor={Wilson, R.~J.},
      series={Encyclopedia of Mathematics and its Applications},
      volume={156},
   publisher={Cambridge University Press},
     address={Cambridge},
       pages={230\ndash 254},
}

\bib{Conrey1988}{article}{
      author={Conrey, B.},
      author={Ghosh, A.},
       title={On the {Z}eros of the {T}aylor {P}olynomials {A}ssociated with
  the {E}xponential {F}unction},
        date={1988-06},
     journal={The American Mathematical Monthly},
      volume={95},
      number={6},
       pages={528},
         url={https://doi.org/10.2307/2322757},
}

\bib{dieudonne1935zeros}{article}{
      author={Dieudonn{\'e}, J.},
       title={Sur les z{\'e}ros des polynomes-sections de $e^x$},
        date={1935},
     journal={Bull. Sci. Math},
      volume={70},
       pages={333\ndash 351},
}

\bib{dohmen1}{article}{
      author={Dohmen, K.},
       title={A {Broken-Circuits-Theorem} for hypergraphs},
        date={1995},
     journal={Archiv der Mathematik},
      volume={64},
       pages={159\ndash 162},
}

\bib{Dohmen1999}{article}{
      author={Dohmen, K.},
       title={An improvement of the inclusion-exclusion principle},
        date={1999},
        ISSN={1420-8938},
     journal={Archiv der Mathematik},
      volume={72},
      number={4},
       pages={298\ndash 303},
}

\bib{dohmen2}{article}{
      author={Dohmen, K.},
       title={An inductive proof of {Whitney's Broken Circuit Theorem}},
        date={2011},
     journal={Discussiones Mathematicae Graph Theory},
      volume={31},
      number={3},
       pages={509\ndash 515},
}

\bib{dohmen-trinks}{article}{
      author={Dohmen, K.},
      author={Trinks, M.},
       title={An abstraction of {Whitney's Broken Circuit Theorem}},
        date={2014},
     journal={The Electronic Journal of Combinatorics},
      volume={21},
      number={4},
       pages={\#P4.32},
}

\bib{Dong2005}{book}{
      author={Dong, F.M.},
      author={Koh, K.M.},
      author={Teo, K.L.},
       title={{Chromatic Polynomials and Chromaticity of Graphs}},
   publisher={World Scientific},
        date={2005},
}

\bib{friedli_velenik_2017}{book}{
      author={Friedli, S.},
      author={Velenik, Y.},
       title={Statistical mechanics of lattice systems: A concrete mathematical
  introduction},
   publisher={Cambridge University Press},
        date={2017},
        ISBN={978-1-107-18482-4},
}

\bib{huh}{article}{
      author={Huh, J.},
       title={Milnor numbers of projective hypersurfaces and the chromatic
  polynomial of graphs},
        date={2012},
        ISSN={0894-0347},
     journal={Journal of the Americam Mathematical Society},
      volume={25},
      number={3},
       pages={907\ndash 927},
}

\bib{Newman1972}{article}{
      author={Newman, D.J.},
      author={Rivlin, T.J.},
       title={The zeros of the partial sums of the exponential function$\ast$},
        date={1972-06},
     journal={Journal of Approximation Theory},
      volume={5},
      number={4},
       pages={405\ndash 412},
}

\bib{Newman1976}{article}{
      author={Newman, D.J.},
      author={Rivlin, T.J.},
       title={Correction: {T}he zeros of the partial sums of the exponential
  function},
        date={1976-04},
     journal={Journal of Approximation Theory},
      volume={16},
      number={4},
       pages={299\ndash 300},
}

\bib{Pfister1991}{article}{
      author={Pfister, C.E.},
       title={Large deviations and phase separation in the two-dimensional
  {I}sing model},
        date={1991},
     journal={Helvetica Physica Acta},
      volume={64},
      number={7},
       pages={953\ndash 1054},
}

\bib{Pritsker1997}{article}{
      author={Pritsker, I.E.},
      author={Varga, R.S.},
       title={Zero distribution, the {S}zeg\"o curve, and weighted polynomial
  approximation in the complex plane},
        date={1997-10},
     journal={Transactions of the American Mathematical Society},
      volume={349},
      number={10},
       pages={4085\ndash 4106},
}

\bib{read}{article}{
      author={Read, R.C.},
       title={An introduction to chromatic polynomials},
        date={1968},
        ISSN={0021-9800},
     journal={Journal of Combinatorial Theory},
      volume={4},
      number={1},
       pages={52\ndash 71},
}

\bib{Scott2005}{article}{
      author={Scott, A.D.},
      author={Sokal, A.D.},
       title={{The Repulsive Lattice Gas, the Independent-Set Polynomial, and
  the Lov\'asz Local Lemma}},
        date={2005-03},
     journal={Journal of Statistical Physics},
      volume={118},
      number={5-6},
       pages={1151\ndash 1261},
}

\bib{szego1924eigenschaft}{article}{
      author={Szeg{\"o}, G.},
       title={{\"U}ber eine eigenschaft der exponentialreihe},
        date={1924},
     journal={Sitzungsber. Berl. Math. Ges},
      volume={23},
       pages={50\ndash 64},
}

\bib{trinks}{article}{
      author={Trinks, M.},
       title={A {N}ote on a {Broken-Cycle} {T}heorem for {H}ypergraphs},
        date={2014},
     journal={Discussiones Mathematicae Graph Theory},
      volume={34},
      number={3},
       pages={641\ndash 646},
}

\bib{ueltschi2004cluster}{article}{
      author={Ueltschi, D.},
       title={Cluster expansions and correlation functions},
        date={2004},
     journal={Moscow Mathematical Journal},
      volume={4},
      number={2},
       pages={511\ndash 522},
}

\bib{varga2008dynamical}{article}{
      author={Varga, R.S.},
      author={Carpenter, A.J.},
      author={Lewis, B.W.},
       title={The dynamical motion of the zeros of the partial sums of $e^z$,
  and its relationship to discrepancy theory},
        date={2008},
     journal={Electron. Trans. Numer. Anal},
      volume={30},
       pages={128\ndash 143},
}

\bib{voloshin1993}{article}{
      author={Voloshin, V.~I.},
       title={The mixed hypergraphs},
        date={1993},
        ISSN={1561-4042},
     journal={Computer Science Journal of Moldova},
      volume={1},
      number={1},
       pages={1},
}

\bib{Walker2003}{article}{
      author={Walker, P.},
       title={The {Z}eros of the {P}artial {S}ums of the {E}xponential
  {S}eries},
        date={2003-04},
     journal={The American Mathematical Monthly},
      volume={110},
      number={4},
       pages={337},
}

\bib{whitney}{article}{
      author={Whitney, H.},
       title={{A Logical Expansion in Mathematics}},
        date={1932},
     journal={Bulletin of the American Mathematical Society},
      volume={38},
      number={8},
       pages={572\ndash 579},
}

\bib{Zemyan2005}{article}{
      author={Zemyan, S.M.},
       title={On the {Z}eroes of the {N}th {P}artial {S}um of the {E}xponential
  {S}eries},
        date={2005-12},
     journal={The American Mathematical Monthly},
      volume={112},
      number={10},
       pages={891},
}

\bib{Zykov_1974}{article}{
      author={Zykov, A.~A.},
       title={{Hypergraphs}},
        date={1974-12},
     journal={Russian Mathematical Surveys},
      volume={29},
      number={6},
       pages={89\ndash 156},
}

\end{biblist}
\end{bibdiv}

\end{document}